\documentclass[10pt]{IEEEtran}
\IEEEoverridecommandlockouts
\usepackage{hyperref}
\usepackage{cite}
\usepackage{graphicx}
\usepackage{psfrag}
\usepackage[cmex10]{amsmath}
\usepackage{amssymb}
\usepackage{ushort}
\usepackage{array}
\usepackage[tight,footnotesize]{subfigure}
\usepackage{epstopdf}

\usepackage[usenames]{color}
 \newtheorem{thm}{Theorem}
 \newtheorem{lem}[thm]{Lemma}

 \newtheorem{defn}{Definition}
 \newtheorem{exmp}{Example}
 \newtheorem{rem}{Remark}

 \newtheorem{assm}{Assumption}

\newcommand{\bb}[0]{\begin{bmatrix}}
\newcommand{\eb}[0]{\end{bmatrix}}
\newcommand{\be}[0]{\begin{equation}}
\newcommand{\ee}[0]{\end{equation}}
\newcommand{\ben}[0]{\begin{equation*}}
\newcommand{\een}[0]{\end{equation*}}

\newcommand{\norms}[1]{\lVert#1\rVert}

\let\norm=\normB

\newcommand{\norml}[2]{\norm{#1}_{L_{#2}}}

\newcommand{\abs}[1]{\lvert#1\rvert}

\renewcommand{\Re}[0]{\mathbb R}

\DeclareFontFamily{U}{mathx}{\hyphenchar\font45}
\DeclareFontShape{U}{mathx}{m}{n}{<-> mathx10}{}
\DeclareSymbolFont{mathx}{U}{mathx}{m}{n}
\DeclareMathAccent{\widebar}{0}{mathx}{"73}

\begin{document}

\title{Closed--loop Reference Models for Output--Feedback Adaptive Systems}

\author{Travis~E.~Gibson, Anuradha~M.~Annaswamy,  and Eugene~Lavretsky,
\thanks{T.~E. Gibson and A.~M. Annaswamy are with the Department
of Mechanical Engineering, Massaschusetts Institute of Technology, Cambridge,
MA, 02139 USA e-mail: ({tgibson@mit.edu}).}
\thanks{E. Lavretsky is with The Boeing Company, Huntington Beach, CA 92648 USA.}
}

\maketitle

\begin{abstract}
Closed--loop reference models have recently been proposed for states accessible adaptive systems. They have been shown to have improved transient response over their open loop counter parts. The results in the states accessible case are extended to single input single output plants of arbitrary relative degree.\end{abstract}

\section{Introduction}

Recently a class of adaptive controllers with {\em Closed--loop Reference Models} (CRM) for states accessible control has been proposed \cite{lav10aiaa,gib13tranA,gib13acc1,gib13acc2}.  The main feature of this class is the inclusion of a Luenberger gain which feeds back the tracking error into the reference model. Without the Luenberger gain the CRM reduces to the {\em Open--loop Reference Model} (ORM) which is used in classical adaptive control \cite{annbook,ioabook}. Reference \cite{lav10aiaa} introduces the concept of the CRM. In references \cite{gib13tranA,gib13acc1,gib13acc2}  the stability and robustness properties of the CRM based adaptive system, and more importantly, an improved transient response were established for the case when state variables are accessible. The transient response was quantified through the use of $\mathcal L_2$ norms of the model following error as well as the rate of control input. In \cite{gib13tranA,gib13acc1,gib13acc2}, it was shown that the extra design freedom in the adaptive system in the form of the Luenberger gain allowed this improvement. Others recent works on states accessible CRM adaptive control can be found in \cite{ste10,ste11}.
 
This paper addresses the next step in the design of adaptive systems, which is the case when only outputs are available for measurement rather than the entire state. It is shown that even with output feedback, the resulting CRM--based adaptive systems are first and foremost stable, and exhibit an improved transient response. As in the case when states are accessible, it is shown that this improvement is possible due to the suitable choice of the Luenberger gain. Unlike the approach in \cite{lav12tac}, the classical model reference adaptive control structure is used here. Also, our focus here is only on single-input single-output systems. 

Using CRMs has two advantages over ORMs: 1) The reference model need not be {\em Strictly Positive Real} (SPR) for CRM systems, and need only have the same number of poles and zeros as its ORM counter part; 2) In CRM systems the reference model, filters and Luenberger gain can be chosen so that the error transfer function used in the update law is SPR and has arbitrarily fast poles and zeros. While  the stability and performance bounds are given for arbitrary reference models, we show in Examples 1 and 2 how one can explicitly obtain error transfer functions of the form
\be\label{intro}
k\frac{s^{m-1}+b_1 s^{m-2} + \cdots b_{m-1}}{s^m+a_1 s^{m-1} + \cdots+ a_{m}}  \triangleq k \mathcal W^\prime(s)
\ee
where $m$ is the relative degree of the plant to be controlled, $s$ is the differential operator, $k$ is the high--frequency gain which is unknown but with known sign, and the $a_i,b_i$ are free to choose so long as $\mathcal W^\prime (s)$ is SPR. 

Another contribution of this work comes by way of the performance analysis technique used. When studying the stability of output feedback adaptive systems non--minimal state space representations of the model following error are constructed so that it can be shown that {\em all} signals in the system are bounded. After stability is obtained, the performance analysis comes by way of studying the behavior of a minimal representation of the adaptive system. The analysis is no longer hindered by the unknown  eigenvalues of the non--observable states in the error equation. It is precisely this technique that allows us to extend the results of transient response analysis from the states accessible case to the output feedback case, where we will show that we have complete control over the location of the eigenvalues of the minimal system.

This paper is organized as follows. Section II contains the notation. In Section III the control problem is defined. Section IV contains the analysis of the ORM (classical) relative degree 1 case. Section V and VI contain the analysis of the CRM relative degree 1 and 2 cases respectively. Section VII analysis the arbitrary relative degree case, and Section VIII closes with our conclusions.

\section{Notation} All norms unless otherwise stated are the Euclidean norm and enduced Euclidean norm. Let $\mathcal P\mathcal C_{[0,\infty)}$ denote the set of all bounded piecewiese continuous signal.
\begin{defn}
Let ${x,y\in \mathcal P\mathcal C_{[0,\infty)}}$. The big {O--notation}, ${y(t)=O[x(t)]}$ is equivalent to the existence of constants ${M_1, M_2>0}$ and ${t_0 \in \Re^+}$ such that 
${\abs{y(t)}\leq M_1 \abs{x(t)} + M_1}\ \forall t\geq t_0.$
\end{defn}

\begin{defn}
Let $x,y\in \mathcal P\mathcal C_{[0,\infty)}$. The small o--notaion, $y(t)=o[x(t)]$ is equivalent to the existence of constants $\beta(t)\in \mathcal P\mathcal C_{[0,\infty)}$ and $t_0 \in \Re^+$ such that ${\abs{y(t)} =\beta(t) x(t)} \ \forall t\geq t_0$ and $\lim_{t\to\infty}\beta(t)=0$.
\end{defn}
\begin{defn}
Let $x,y\in \mathcal P\mathcal C_{[0,\infty)}$. If $y(t)=O[x(t)]$  and $x(t)=O[y(t)]$. Then $x$ and $y$ are said to be equivalent and denoted as $ x(t)\sim y(t) $.
\end{defn}
\begin{defn}
Let $x,y\in \mathcal P\mathcal C_{[0,\infty)}$. $x$  and $y$ are said to grow at the same rate if
$\sup_{t\leq \tau} \abs{x(\tau)} \sim \sup_{t\leq \tau} \abs{y(\tau)}$.\end{defn}
\begin{defn}\label{def:5}
The prime notation is an operator that removes the high frequency gain from a transfer function \ben
\mathcal W(s) \triangleq k\frac{s^{m-1}+b_1 s^{m-2} + \cdots b_{m-1}}{s^m+a_1 s^{m-1} + \cdots+ a_{m}}. \een
so that 
\ben
\mathcal W^\prime(s) \triangleq \frac{\mathcal W(s)} k,
\een
Just as was done in \eqref{intro}.
\end{defn}

\section{The Control Problem}
 Consider the {\em Single Input Single Output} (SISO) system of equations
\be\label{iop}
y(t)=W(s)u(t)
\ee
where $u\in\Re$ is the input, $y\in\Re$ is the measurable output, and $s$ the differential operator. The transfer function of the plant is parameterized as 
\be\label{eq:w}
W(s)\triangleq k_p\frac{Z(s)}{P(s)} \triangleq k_p W^\prime(s)
\ee
where $k_p$ is a scalar, and $Z(s)$ and $P(s)$ are monic polynomials with ${\text{deg}(Z(s))<\text{deg}(P(s))}$. The following assumptions will be made throughout.

\begin{assm}\label{ass1}
$W(s)$ is minimum phase.
\end{assm}
\begin{assm}\label{ass2}
The sign of $k_p$ is known.
\end{assm}
\begin{assm}\label{ass3}
The relative degree of $W(s)$ is known.
\end{assm}

\section{Classical $n^*=1$ case (ORM $n^*=1$)}

The goal is to design a control input $u$ so that the output $y$ in \eqref{iop} tracks the output $y_m$ of the reference system
\be\label{eq:wm}
y_m(t)=W_m(s)r(t)\triangleq k_m\frac{Z_m(s)}{P_m(s)}r(t)
\ee
where $k_m$ is a scalar and $Z_m(s)$ and $P_m(s)$ are monic polynomials with $W_m(s)$ relative degree 1. Just as before we use the prime notation from Definition \ref{def:5} 
\be
k_m W_m^\prime(s) = W_m(s).
\ee 
\begin{assm}\label{assm4}
$W^\prime_m(s)$ is {\em Strictly Positive Real} (SPR).
\end{assm}
The previous assumption can be relaxed by using pre--filters in the adaptive law, similar to what will be done in the relative degree 2 controller. This increased generalization though is not necessary for our discussion.

The structure of the adaptive controller is now presented:
\begin{align}
\dot \omega_1(t)&=\Lambda \omega_1+b_\lambda u(t) \label{eq:w1}\\
\dot \omega_2(t)&=\Lambda \omega_2+b_\lambda y(t) \label{eq:w2}\\
\omega(t) &\triangleq [r(t),\ \omega_1^T(t),\ y(t),\ \omega_2^T(t) ]^T\\
\theta(t) &\triangleq [ k(t),\ \theta_1^T(t),\ \theta_0(t),\ \theta_2^T(t)  ]^T\\
u&=\theta^T(t)\omega\label{u1}
\end{align}
where $\Lambda \in\Re^{(n-1)\times(n-1)}$ is Hurwitzx, $b_\lambda \in\Re^{n-1}$, $\hat k\in \Re$, $\omega_1,\omega_2\in\Re^{n-1}$, and $\theta\in\Re^{2n}$ is adaptive gain vector with $ k(t)\in\Re$, $\theta_1(t)\in\Re^{n-1}$, $\theta_2(t)\in\Re^{n-1}$ and $\theta_0(t) \in\Re$. The update law for the adaptive parameter is then defined as
\be\label{t1}
\dot \theta(t)=-\gamma\text{sign}(k_p)e_y\omega,
\ee
where $e_y=y-y_m$.

Before stability is proved, a discussion on parameter matching is needed. Let 
$
\theta_c \triangleq [ k_c,\ \theta_{1c}^T,\ \theta_{0c},\ \theta_{2c}^T  ]^T
$
be a constant vector. When $\theta(t)=\theta_c$ the forward loop and feedback loop take the form
\ben
\frac{\lambda(s)}{\lambda(s)-C(\theta_c;s)}\ \text{ and }\ \frac{D(\theta_c;s)}{\lambda(s)}.
\een
For simplicity we choose $\lambda(s)=Z_m(s)$, but note that this is not necessary and the stability of the adaptive system will still hold. The closed loop system is now of the form
\ben
y(t)=W_{cl}(\theta_c; s) r(t)
\een
with\ben
W_{cl}(\theta_c; s)\triangleq  \frac{k_c k_pZ(s)Z_m(s)}{(Z_m(s)-C(\theta_c;s))P(s)- k_p Z(s) D(\theta_c;s)}.
\een
From the Bezout Identity, a  $\theta^{*T}\triangleq [ k^*,\ \theta_{1}^{*T},\ \theta_{0}^*,\ \theta_{2}^{*T}  ]^T$ exists such that
$
{W_{cl}(\theta^*; s)  = W_m(s)}.
$

Therefore, 
\be\label{yptf1}
y(t)=k_p W^\prime_m(s) (\phi^T(t)\omega(t)+k^* r(t))
\ee
and
\be\label{e1}
e_y(t) = k_p W^\prime_m(s) \phi(t) \omega(t),
\ee
where $\phi(t)=\theta(t)-\theta^*(t)$ and $k^*=k_m/k_p$.

\subsection{Stability for $n^*=1$}
The plant in \eqref{eq:w} can be represented by the unknown quadruple, ${(A_p,b_p,c_p,k_p)}$ 
\be\label{eq:wp}
\dot x= A_{p} x+ b_{p} u ; \quad y = k_p c^T_{p}x
\ee
where 
\ben
k_pc_{p}^T (sI- A_{p}) b_{p} = W(s).
\een
In general one does not need to keep the high frequency gain as a separate variable when writing the transfer function dynamics in state space form. In the context of adaptive control however, the sign of $k_p$ is important in proving stability and is therefore always singled out from the rest of the dynamics. Using \eqref{eq:wp}, the dynamics in \eqref{yptf1} can be represented as
\be\label{eq:x}
\dot x = A_{mn} x + b_{mn} (\phi^T(t)\omega + k^* r); \quad y = k_p c^T_{mn}x
\ee
where
\ben\begin{split}
A_{mn} &= \bb A_p + b_p \theta_0^* k_pc_p^T & b_p \theta_1^{*T} & b_p \theta_2^{*T}  \\ b_\lambda \theta_0^* k_pc_p^T & \Lambda+b_\lambda\theta_1^{*T}\ & b_\lambda \theta_2^{*T}\\ b_\lambda k_pc_p^T& 0 & \Lambda\eb \\
b_{mn} &= \bb b_p \\ b_\lambda\\0 \eb, \quad c_{mn} =\bb c_p \\ 0 \\ 0\eb\text{ and  } \ x\triangleq\bb x_p \\ \omega_1 \\ \omega_2\eb
\end{split}\een
with the reference model having an equivalent non--minimal representation
\ben
\dot x_{mn} = A_{mn} x_{mn} + b_{mn} k^* r; \quad y_m = k_p c^T_{mn}x_{mn}
\een
with the property that
\ben
k_pc_{mn}^T (sI- A_{mn}) b_{mn} =k_p W^\prime _m(s).
\een
The non--minimal error vector is defined as ${e_{mn}=x-x_{mn}}$ and satisfies the following dynamics
\be
\dot e_{mn} = A_{mn} e_{mn} + b_{mn} \phi^T \omega ; \quad  e_y = k_pc^T_{mn}e_{mn}.
\ee

\begin{thm} Following Assumptions \ref{ass1}-\ref{assm4}, the plant in \eqref{iop} with the reference model in \eqref{eq:wm}, controller in \eqref{u1} and the update law in \eqref{t1} are globally stable with the model following error asymptotically converging to zero.\end{thm}
\begin{proof} See \cite[\S 5.3]{annbook}.
\end{proof}

\section{CRM $n^*=1$}\label{s:crm1}
In the case of ORM adaptive control, the reference model only receives one input and is unaffected by the plant state trajectory. In order to facilitate the use of a Luenbereger feedback gain $\ell$ into the reference model, the reference model is chosen as\be\label{crm1:ref}
\dot x_m = A_m x_m +b_m k_m r + \ell(y-y_m), \quad y_m= c_m^T x_m
\ee
where $(A_m,b_m,c_m^T)$ is an $m$ dimensional system in observer canonical form
with ${c_m^T=[0 \ \ldots \ 0 \ 1]}$ and satisfying
\ben
c_{m}^T (sI- A_{m}) b_{m} k_m =  W_m(s).
\een
$y_m(t)$ is now related to the reference command $r(t)$ and model following error $e_y(t)$ as
\be\label{ym:oc}
y_m(t) = W_m(s) r(t) + W_\ell(s) (y(t)-y_m(t))
\ee
where
\be\label{wl}
W_\ell (s) \triangleq k_\ell \frac{Z_\ell(s)}{P_m(s)},
\ee
and $k_\ell \in \Re$ along with the $m-1$ order monic polynomial $Z_\ell(s)$ are a function of $\ell$ and free to choose. Subtracting \eqref{ym:oc} from \eqref{yptf1} results in the following differential relation
\be \label{e2}e_y = k_p W^\prime_e(s) \phi^T  \omega
\ee
where
\be\label{epmzl}
W^\prime_e(s)\triangleq \frac{Z_m(s)}{P_m(s)-k_\ell Z_\ell(s)}.
\ee

\begin{lem}\label{lem:ell}
An $\ell$ can be chosen such that ${W^\prime_e(s)}$ is SPR for any ${n^*=1}$ and minimum phase transfer function ${W^\prime_m(s)}$.
\end{lem}
\begin{proof}
The product $k_\ell Z_\ell(s)$ a polynomial of order ${n-1}$ with ${n-1}$ degrees of freedom through $\ell$. $P_m(s)$ is a monic polynomial of degree $n$. Therefore, ${P_m(s)-k_\ell Z_\ell(s)}$ is a monic polynomial of order $n$ with $n-1$ degrees of freedom determined by $\ell$. Thus for any $Z_m(s)$ the roots of $W^\prime _e(s)$ can be placed freely in the closed left--half plane such that $W^\prime_e(s)$ is SPR.
\end{proof}

Let
\be
A_e=A_{mn}+G\ell k_p c^T_{mn}
\ee
where $G$ transforms $x_m$ to the controllable subspace in $x_{mn}$, which always exist \cite{kal:siam63}. The non--minimal error dynamics therefore take the form
\be
\dot e_{mn}(t) = A_e e_{mn}(t) + b_{mn} \phi(t) \omega(t).
\ee
\begin{rem}
It is worth noting that in the construction of the minimal and non--minimal systems the location of the gains $k_p$ and $k_m$ switch from being located at the input to the output. The non--minimal systems is never created and thus need not be realized. Therefore, the influence of $k_p$ whether it be on the input or output matrix of the state space does not matter. For the case of the minimal reference model in \eqref{crm1:ref} it is critical however that $k_m$ appears at the input of the system. This is done on purpose so that given the canonical form of $c_m$ the $\ell$ in \eqref{crm1:ref} completely determines the zeros and high frequency gain of $W_\ell(s)$ in \eqref{wl}.\end{rem}

\begin{thm} Following Assumptions \ref{ass1}-\ref{ass3} and $\ell$ chosen as in Lemma \ref{lem:ell}, the plant in \eqref{iop} with the reference model in \eqref{crm1:ref}, controller in \eqref{u1} and the update law in \eqref{t1} are globally stable with the model following error asymptotically converging to zero.
\end{thm}
\begin{proof} Given that $W^\prime_e(s)$ is SPR, there exists a ${P_e=P_e^T>0}$ such that
\be
A_{e}^T P_e+P_e A_{e}= - Q_e \text{ and } P_e b_{mn} = c_{mn}.
\ee
where ${Q_e=Q_e^T>0}$. Thus
\be\label{lyap}
V=e_{mn}^T P_{e}e_{mn} + \frac{\phi^T \phi}{\gamma \abs{k_p}}
\ee
is a Lyapunov function with derivative $\dot V=-e_{mn}^T Q_{e}e_{mn}$. Barbalat Lemma ensures the asymptotic convergence of $e_{mn}$ to zero.
\end{proof}

\subsection{Performance}
Now that we have proved stability we can return to a minimal representation of the error dynamics in \eqref{e2} which is  
\be
\dot e_m = A_{\ell} e_m + b_m k_p \phi^T \omega, \quad e_y = c_m^T e_m;
\ee
where the all the eigen--values of $A_\ell$ are the roots to  ${P_m(s)-k_\ell Z_\ell(s)}$, as can be seen from \eqref{epmzl}. Recall the Anderson version of KY Lemma;
\be
A_\ell^T P + P A_\ell = -gg^T - 2\mu P; \quad
Pb_m = c_m
\ee
where 
\be\label{eq:mu}
\mu\triangleq \min_i\left| \lambda_i(A_\ell)\right|,\quad i=1\text{ to }m.
\ee
The following performance function 
\be
V_p=e_m^TPe_m +  \frac{\phi^T \phi}{\gamma \abs{k_p}}
\ee
has a time derivative
\be\label{crm1:vp}
\dot V_p \leq -2 \mu e_m^TPe_m.
\ee
From \eqref{crm1:vp} it directly follows that
\be\label{ynormstar1}
\norml{e_y(t)}{2}^2\leq \frac{1}{2\mu} \left( \frac{\lambda_\text{max}(P)}{\lambda_\text{min}(P)}\norm{e(0)}^2+\frac{1}{\gamma \abs{k_p}}\frac{\norm{\phi(0)}^2}{\lambda_{\text{min}}(P)}\right).
\ee

\begin{exmp} \label{ex:11} The transfer function $W^\prime_e(s)$ must be SPR, therefore, the poles of $W^\prime_e(s)$ are limited by the location of its zeros. The order of $A_m$ however is free to choose so long as $m\geq1$, thus we can choose $m=1$. Therefore making 
\ben
W_m(s)= k_m \frac{1}{s+a_m}
\een
where $b_m=k_m$ and $A_m=-a_m$. The closed loop reference model transfer function therefore is 
\be\label{ex:1}
W_e(s)=k_m\frac{1}{s+a_m+l}
\ee
where ${\ell=-l}$, ${l>0}$. From \eqref{ex:1}, it is clear that there are no zeros limiting the location of the closed loop pole. 

Further more, the Anderson Lemma reduces to the trivial solution of $P=1$, $g=0$, and $\mu=a_m+l$. Since there are no zeros to worry about $W^\prime _e(s)$ is SPR for all $l$. Therefore, $\mu$ can can be chosen arbitrarily. The bound in \eqref{ynormstar1} for this example simplifies to
\be\label{ex1ey}
\norml{e_y(t)}{2}\leq \frac{1}{2(a_m+l)} \left( \norm{e(0)}^2+\frac{\norm{\phi(0)}^2}{\gamma\abs{k_p}}\right).
\ee
\end{exmp}

\begin{rem}
The use of CRMs has two advantages compared to the use of ORMs. The first is that the reference model need not be SPR a priori, but only needs to be of appropriate relative degree. There are several methods of dealing with non--SPR reference models for ${n^*=1}$, but these methods require the use of pre--filters \cite{kru02}, or augmented error approaches (see \cite{annbook}, and Section \ref{arbstar}).

The second advantage is illustrated in Example \ref{ex:11}. Using this approach, a reference model can be chosen such that it has no zeros. When this is done and a CRM is used, the location of the slowest pole of the error model dynamics is free to choose. When using ORMs, the location of the slowest eigenvalue of the closed--loop error model is not free to choose, as speeding up the reference model eigenvalues without the use of CRMs will require the use of high--gain feedback which is equivalent to $\norm{\theta^*}$ being large if the open--loop plant has slow eigenvalues.
\end{rem}

\section{CRM SISO $n^*=2$}
Consider the dynamics in \eqref{iop} where the relative degree of the transfer function in \eqref{eq:w} is now  2 instead of 1 and the reference to be followed is the CRM in \eqref{crm1:ref}. The control input in \eqref{u1} will no  longer lead to stable adaptation and must be adjusted as
\begin{align}
u(t)=&\dot\theta^T(t) \zeta(t) + \theta^T(t) \omega(t) \\
\dot\theta(t) =& - \text{sign}(k_p)e_y(t) \zeta(t)^T
\end{align}
where $\zeta(t)$ is a filtered version of the regressor vector $\omega$ and defined as \be
\zeta (t) = A^{-1}(s) \omega(t) \text{ where }  
A(s) = s+ a.\label{As}
\ee
Using the same reference model as in \eqref{crm1:ref}, the error $e_y(t)$ now takes the form
\be \label{e:2}e_y(t) =  k_p W^\prime _e(s) A(s) \phi^T(t)  \zeta(t).
\ee
With $\ell$ and $A(s)$ chosen such that the transfer function $W^\prime_e(s) A(s)$ is SPR the CRM adaptive controller for $n^*=2$ is stable.

\subsection{Performance}
The same analysis performed in the previous section can be used to analyze the ${n^*=2}$ case. The minimum eigenvalue of $W^\prime_e(s)A(s)$ in \eqref{e:2} along with $\gamma$ control the $\mathcal L_2$ norm of $e_y$. As in the previous example, a reference model with no zeros that is relative degree 2 can be chosen. Then, the zeros of $W^\prime_e(s)A(s)$ are completely determined by $A(s)$ and the poles are freely placed with $\ell$. Thus any SPR transfer function of order 2 can be created with an arbitrarily fast slowest eigenvalue.

\section{CRM Arbitrary $n^*$}\label{arbstar}
The adaptive controller for $n^*=2$ is special given that we have access to $\dot\theta(t)$. Instead, for higher relative degrees it is common to use an augmented error approach, where by the original model following error $e_y$ is not used to adjust the adaptive parameter, but an augmented error signal which does satisfy the SPR conditions needed for stability. The augmented error method used in this result is Error Model 2 as presented in \cite[\S 5.4]{annbook}, with some changes to the notation.

For ease of exposition and clarity in presentation we present the $k_p$ known and $k_p$ unknown presentation in two sections.
\subsection{Stability for  known high frequency gain}
We begin by replacing Assumption 2 with:

{\it Assumption 2$^\prime$}: $k_p$ is known. \\
\noindent Without loss of generality we choose $k_m=k_p=1$ and the control input for the generic relative degree case reduces to
\be\label{un}
u(t)=r(t)+\widebar\theta^T(t) \widebar\omega(t)
\ee
where $\widebar{(\cdot)}$ denotes the vectors,
\begin{align}
\widebar\omega(t) &\triangleq [\omega_1^T(t),\ y(t),\ \omega_2^T(t) ]^T \label{wbaro}\\
\widebar\theta(t) &\triangleq [ \theta_1^T(t),\ \theta_0(t),\ \theta_2^T(t)  ]^T.
\end{align}
A feedforward time varying adaptive gain $k(t)$ is no longer needed and thus $r(t)$ has been removed from the regressor vector do to the fact that $k_p=k_m=1$. The model following error then, satisfies the following differential relation
\be \label{e3}e_y =  W^\prime_e(s) \widebar\phi^T  \widebar\omega
\ee
where the reader is reminded that the prime notation removes the high frequency gain from transfer functions, and since $k_m=k_p=1$, $W^\prime_e(s) = W_e(s)$. Similar to the use of $A(s)$ in \eqref{As} for the relative degree 2 case, a stable minimally realized filter $F(s)$ with no zeros is used to generate the filtered regressor
\be\label{wz}
\widebar\zeta = F(s) I \widebar \omega
\ee
where $I$ is the $2n-1$ by $2n-1$ identity matrix, $F(s)$ designed with unity high frequency gain, and $F(s)$ and $\ell$ chosen so that
\be
W^\prime_f(s)  \triangleq W^\prime_e(s)F^{-1}(s)
\ee
is SPR.
\begin{lem}
For any stable $F(s)$ an $\ell$ can be chosen such that $W^\prime_f(s)$ is SPR.
\end{lem}
\begin{proof} The proof follows the same arguments as in Lemma \ref{lem:ell}.\end{proof}

The tuning law for the arbitrary relative degree case uses an augmented error $e_a$, which is generated from the model following error $e_y$ and an auxiliary error $e_\chi$. Using the CRM in \eqref{crm1:ref}, the augmented and auxiliary error are defined as:
\begin{align}
e_a& \triangleq e_y + W^\prime_f(s) \left( e_\chi  - e_a\widebar \zeta^T \widebar\zeta\right) \label{ea}\\
e_\chi& \triangleq \widebar\theta^T \widebar\zeta - F(s)\widebar\theta^T \widebar\omega\label{eu}.
\end{align}
A stable tuning law for the system is then defined as
\be\label{tdbar}
\dot{\widebar\theta} = -\gamma e_a \bar\zeta.
\ee

\begin{thm}
Following Assumptions 1, 2$^\prime$ and 3, with $\ell$ chosen such that $W^\prime_f(s)$ is SPR, the plant in \eqref{iop} with the reference model in \eqref{crm1:ref}, controller in \eqref{un} and update law in \eqref{tdbar} are globally stable with the model following error $e_y$ asymptotically converging to zero.
\end{thm}
\begin{proof} The proof proceeds in 4 steps.  First it is shown that $\widebar\theta(t)$ and $e_a$ are bounded and that ${e_a,\dot{\widebar\theta} \in \mathcal L_2}$. Second, treating $\widebar\theta(t)$ as a bounded time--varying signal, then all signals in the adaptive system can grow at most exponentially. Third, if it is assumed that the signals grow in an unbounded fashion, then it can be shown that $y$, $\omega_1$ $\omega_2$, $\widebar\omega$, $\widebar\zeta$ and $u$ grow at the same rate. Finally, from the fact that ${\dot{\widebar\theta}\in \mathcal L_2}$ it is shown that $\omega_2$ and $\widebar\omega$ do not grow at the same rate. This results in a contradiction and therefore, all signals are bounded and furthermore, $e_y(t)$ asymptotically converges to zero. Steps 1 and 4 are detailed below. Steps 1-3 follow directly from \cite[\S 5.5]{annbook} with little changes. Step 4 does involve a modification to the analysis which is addressed in detail next.
\subsubsection*{Step 1}
Expanding the error dynamics in \eqref{ea} and canceling like terms of $W_e^\prime(s)\widebar\theta^T\omega$ we have
\ben
e_a = -W^\prime_e(s) \widebar\theta^{*T} \widebar\omega+W^\prime_f(s) \left( \widebar \theta^T \widebar\zeta  - e_a\widebar \zeta^T \widebar\zeta\right). 
\een
Adding and subtracting $W^\prime_f(s) \widebar\theta^{*T} \widebar\zeta$ the equation becomes
\be\label{eq:ea2}
e_a = W^\prime_f(s) \left( \widebar \phi^T \widebar\zeta  - e_a\widebar \zeta^T \widebar\zeta\right)  + \delta(t)
\ee
where $\delta(t)$ is an exponentially decaying term do to initial conditions and defined as
\be
\delta(t) = W^\prime_f(s)\left(\widebar\theta^{*T} \widebar\zeta(t) - F(s)\widebar\theta^{*T} \widebar\omega(t)\right).
\ee
Breaking apart $\widebar\zeta$ from its definition in \eqref{wz} and noting that $\widebar\theta^*$ now commutes with $F(s)$ we have that
\be
\delta(t)= W^\prime_f(s)\left(\widebar\theta^{*T} \left( F(s)  - F(s)\right) I \widebar\omega\right).
\ee
Therefore, if the filter $F(s)$ is chosen to have the same initial conditions when constructing $\widebar\zeta$ and $e_\chi$ then, $\delta=0$ for all time.  For this reason we ignore the affect of choosing different filter initial conditions. The interested reader can see how one can prove stability in augmented error approaches where $\delta(0)\neq 0$ \cite[pg. 213]{annbook}, with the addition of an extra term in the  Lyapunov function. 

A non--minimal representation of $e_a$ is given as
\be\label{eanm}
\dot e_{an} = A_{e} e_{an} + b_{an}\left( \widebar\phi^T \widebar \zeta - e_a \widebar\zeta^T \widebar\zeta\right), \quad e_a = c^T_{an} e_{an}
\ee
where
\be
c_{an}^T(sI-A_e)^{-1} b_{an} \triangleq W^\prime_f(s).
\ee
Given that $W_f(s)$ is SPR, there exists a ${P_a=P_a^T>0}$ such that
\be
A_e^T P_a+P_a A_{e}= - Q_a \text{ and } P_a b_{an} = c_{an}.
\ee
where ${Q_a=Q_a^T>0}$.

Consider the Lyapunov candidate
\be\label{lyap}
V=e_{an}^T P_{a}e_{an} + \frac{\phi^T \phi}{\gamma}
\ee
Differentiating along the system dynamics in \eqref{eanm} and substitution of the tuning law from \eqref{tdbar} results in
\be
\dot V \leq - e_{an}^TQ_a e_{an} - 2  e_a^2 \widebar\zeta^T \widebar\zeta.
\ee
Therefore, $e_{an},\widebar\theta \in \mathcal L_\infty$ and $e_{an},\dot{\widebar\theta} \in\mathcal L_2$

\subsubsection*{Step 2}  The plant dynamics can be expressed as
\be
\dot x = A_{mn} x + b_{mn} (\widebar\phi^T(t)\omega + r); \quad y = c^T_{mn}x
\ee
where with an appropriate choice of a $C$ can be expressed as
\be
\dot x = \left(A_{mn} +  b_{mn} \widebar\phi^T(t)C \right) x + b_{mn} r
\ee
From Step 1 it is known that $\widebar\phi$ is bounded, and therefore $x$ grows at most exponentially. Futhermore, for $r$ piecewise continuous, $x$ and $\widebar\zeta$ are both piecewise continuous as well.
\subsubsection*{Step 3} If it is assumed that all signals grow in an unbounded fashion then it can be shown that
\be\label{sim}
\begin{split}
\sup_{\tau\leq t} \abs{y(\tau)}\sim \sup_{\tau\leq t}\norm{\omega_1(\tau)} \sim \sup_{\tau\leq t}\norm{\omega_2(\tau)}  \ldots \\ \sim \sup_{\tau\leq t}\norm{\widebar\omega}\sim \sup_{\tau\leq t}\norm{\widebar\zeta}\sim \sup_{\tau\leq t} \abs{u(\tau)} 
\end{split}
\ee
\cite[\S 5.5]{annbook}
\subsubsection*{Step 4} Rewinting \eqref{eu} in terms of $\widebar\omega$ we have that
\be\label{chibound}
e_\chi \triangleq \widebar\theta^T F(s) I \widebar\omega - F(s)\widebar\theta^T \widebar\omega
\ee
and given that $\dot{\widebar\theta}\in\mathcal L_2$ and $F(s)$ is stable the following holds
\be\label{euo}
e_\chi(t) = o\left[\sup_{\tau\leq t} \norm{\widebar\omega(\tau)}\right].
\ee
The above bound follows from the {\em Swapping Lemma} \cite[Lemma 2.11]{annbook}.
From \eqref{tdbar} and the fact that $\dot{\widebar\theta} \in\mathcal L_2$ we have that ${e_a\widebar\zeta} \in\mathcal L_2$. Given that $W^\prime_f(s)$ is asymptotically stable, \cite[Lemma 2.9]{annbook} can be applied and it follows that
\be\label{wfso}
W^\prime_f(s) \left( (e_a\widebar\zeta)^T \widebar\zeta  \right) = o \left[\sup_{\tau\leq t}  \norm{\widebar\zeta(\tau)} \right]
\ee
The plant output can be written in terms of the reference model and model following error as
\ben
\begin{split}
y(t)= &y_m(t) + e_y(t)  \\
=& W^\prime_m(s) r(t) + \left(1+W^\prime_\ell(s) \right) e_y(t).
\end{split}\een
Using \eqref{ea}, ${e_y(t) = e_a - W^\prime_f(s) \left( e_\chi  - e_a\widebar \zeta^T \widebar\zeta\right)}$ and the above equation expands as
\ben
\begin{split}
y(t)
=& W^\prime_m(s) r(t) + \left(1+W^\prime_\ell(s) \right) e_a \\ &- \left(1+W^\prime_\ell(s) \right)W^\prime_f(s) \left( e_\chi  - e_a\widebar \zeta^T \widebar\zeta\right).
\end{split}
\een
Using \eqref{euo} \eqref{wfso} and noting that $1+W^\prime_\ell(s)$ is asymptotically stable \cite[Lemma 2.9]{annbook} can be applied again and
\ben
\begin{split}
y(t) =& W^\prime_m(s) r(t) + \left(1+W^\prime_\ell(s) \right) e_a \\ &+ o\left[\sup_{\tau\leq t} \norm{\widebar\zeta(\tau)}\right]+ o\left[\sup_{\tau\leq t} \norm{\widebar\omega(\tau)}\right].\end{split}\een
Given that $r$ and $e_a$ are piecewise continuous and bounded we finally have that
\be
y(t) = o\left[\sup_{\tau\leq t} \norm{\widebar\omega(\tau)}\right].
\ee
This contradicts \eqref{sim} and therefore all signals are bounded. Furthermore, from \eqref{eanm} it now follows that $\dot e_{an}$ is bounded and given that $e_{an}\in \mathcal L_2$, from Step 1, it follows that $e_{an}$ asymptotically converges to zero and therefore $\lim_{t\to\infty}e_a(t)=0$. From \eqref{euo} it follows that $e_\chi$ asymptotically converges to zero. Therefore, $\lim_{t\to\infty} e_y(t) = 0$. The above analysis differs from the analysis for the ORM output feedback adaptive control do to the fact that one can not a priori assume that $y_m(t)$ is bounded, do to the feedback of $e_y$ into the reference model.
\end{proof}

\subsection{Performance when $k_p$ known}
Just as in the ${n^*=1}$ case, with stability proved a Lyapunov performance function can be studied that uses a minimal representation of the dynamics. That being said, consider the minimal representation of the dynamics in \eqref{eq:ea2}
\be
\dot e_{am} = A_{\ell} e_{am} + b_{am} \left( \widebar\phi^T \widebar \zeta - e_a \widebar\zeta^T \widebar\zeta\right), \ \ \ e_y = c_{am}^T e_{am}
\ee
in observer canonical form so that ${c_{am}^T=[0 \ \ldots \ 0 \ 1]}$
and \ben c_{am}^T(sI-A_{\ell})^{-1} b_{am} \triangleq W^\prime_f(s)\een

Recall the Anderson version of KY Lemma;
\be
A_\ell^T P_p + P_p A_\ell = -gg^T - 2\mu P_p; \quad
P_pb_{am} = c_{am}
\ee
where $\mu$ is defined in \eqref{eq:mu}. The following performance function 
\be
V_p=e_{am}^TP_pe_{am} +\frac{ \widebar\phi^T\widebar\phi}{\gamma}
\ee
has a time derivative
\be\label{to}\dot V_p \leq -2 \mu e_{am}^TP_pe_{am} -2 e_a^2 \widebar\zeta^T \widebar\zeta.\ee 
From \eqref{to} it directly follows that
\be\label{eanorm}
\norml{e_a(t)}{2}^2\leq \frac{1}{2\mu} \left( \frac{\lambda_\text{max}(P_p)}{\lambda_\text{min}(P_p)}\norm{e(0)}^2+\frac{1}{\gamma}\frac{\norm{\widebar\phi(0)}^2}{\lambda_{\text{min}}(P_p)}\right) \ee
and
\be\label{tdot}
\norml{\dot{\widebar\theta}(t)}{2}^2\leq  \frac{1}{2} \left( \gamma^2 \lambda_\text{max}(P_p)\norm{e(0)}^2+\gamma \norm{\widebar\phi(0)}^2 \right).
\ee
Ultimately we would like to compute the $\mathcal L_2$ norm of $e_\chi$ and $e_y$. Given that these norms will depend explicitly on the specific values of the filter and reference model, we perform that analysis in the following example.

\begin{exmp} In this example we consider a relative degree 2 plant. The reference model is chosen as 
\be
W_m(s) = \frac{1}{s^2+ b_1 s + b_2}
\ee
and the filter is chosen as
\be\label{defs}
F(s) = \frac{1}{s+f_1}.
\ee
The reference model gain is expanded as
\be
\ell = \bb -l_1 & -l_2 \eb^T.
\ee
Then
\be
W_e(s)  = \frac{1}{s^2+ (b_1+l_1) s + (b_2+l_2)}\ee
and
\be\label{above11}
W_f(s)  = \frac{s+f_1}{s^2+ (b_1+l_1) s + (b_2+l_2)}.
\ee
Since, ${k_p=k_m=1}$, then ${W_m(s) = W_m^\prime(s)}$, ${W_e(s) = W_e^\prime(s)}$ and ${W_f(s) = W_f^\prime(s)}$. For stability to hold $W^\prime_f(s)$ must be SPR and from \eqref{above11} it is clear that the SPR condition can be satisfied by choosing $\ell$ and $f_1$ appropriately. More importantly though, we see that the slowest eignvalue of $W_f(s)$ can be arbitrarily placed and thus the $\mu$ in \eqref{eq:mu} can be arbitrarily increased. 

\be\label{normchi2}
\norm{e_\chi(t)}_{\mathcal L_2}^2 \leq 3\left(\frac{e_\chi^2(0)}{2 f_1}+ \left( \frac{e^2_\chi(0)}{4 f_1^2}    + \frac{ \norm{\widebar\omega(t)}_\infty^2 }{f_1^3}\right) \norm{\dot{\bar\theta}(t)}_{\mathcal L_2}^2   \right) 
\ee
A detailed proof of this expression is given in Appendix \ref{app1}. Furthermore, we have the following bound for the model following error
\be
\norm{e_y(t)}^2_{\mathcal L_2} \leq 2 \norm{e_a(t)}_{\mathcal L_2}^2 + 2 \norm{e_\zeta(t)}_{\mathcal L_2}^2 
\ee
where
\be\label{echi}
e_\zeta(t) \triangleq W_f(s)e_\chi(t)
\ee
can be bounded as
\be\label{brefer}
 \norm{e_\zeta}^2_{\mathcal L_2}  \leq 3 m^2 \left(\frac{e_\zeta^2(0)}{2 \mu} + \left( \frac{e_\chi(0)^2}{4 \mu f_1} + \frac{ \norm{\widebar\omega(t)}_\infty^2 }{\mu f_1^2} \right )\norm{\dot{\bar\theta}(t)}_{\mathcal L_2}^2\right).
\ee
The bound in \eqref{brefer} is given in Appendix \ref{app2}.

\end{exmp}

\begin{rem}
     Now we compare the norms in \eqref{normchi2} and \eqref{brefer} for an ORM and CRM system and note that increasing both $f_1$ and $\mu$ decreases the two norms. For the ORM system $\ell=0$, therefore $\mu$ is solely a function of $b_1$ and $b_2$ in \eqref{above11}. The coefficients $b_1$ and $b_2$ can not be arbitrarily changed without affecting the matching parameter vector $\bar \theta^*$. In the presence of persistence of excitation, $\bar\theta(t) \rightarrow \bar\theta^*$ and large $\bar\theta^*$ will directly imply a large control input. Furthermore, one can not arbitrarily change the reference model poles, as the reference model is a target behavior for the plant, in which case the control engineer may not want to track a reference system with arbitrarily fast poles. Therefore, given that $b_1$ and $b_2$ are not completely free to choose this also limits the value of $f_1$ as $W_f(s)$ must always be SPR. In the CRM case $b_1$ and $b_2$ can be held fixed and $l_1$, $l_2$ and $f_1$ can be adjusted so that the poles of $W_f(s)$ are arbitrarily fast and $W_f(s)$ is still SPR.Therefore, the added degree of freedom through $\ell$ in the CRM adaptive systems allows more flexibility in decreasing the $\mathcal L_2$ norm of $e_y$.
\end{rem}

\begin{rem}
In the above, we have derived bounds on the $\mathcal L_2$ norm of the tracking error. That the same error has finite $\mathcal L_\infty$ bounds is easily shown using Lyapunov function arguments and the fact that projection algorithms ensure
exponential convergence of the error to a compact set, similar to the analysis in \cite{gib13tranA,gib13acc1,gib13acc2}.
\end{rem}

\subsection{Stability in the case of unknown high frequency gain}

When $k_p$ is unknown but with known sign as in Assumption 2, the control structure must include $k(t)$ into the adaptive vector as well as including $r(t)$ back into the regressor vector. Therefore, the controller take the form of \eqref{u1}, repeated here in for clarity,
\ben u(t) =\theta^T(t)\omega(t).\een  
The reference model is chosen as in \eqref{crm1:ref} where
$W_m(s)$ has the same relative degree as the plant to be controlled and thus the output error is the same as in \eqref{} but repeated for clarity
\ben
e_y(t) = k_p W^\prime_e(s) \phi^T(t) \omega(t)\een
where $W_e(s)$ is of the same relative degree as the plant. A complete filtered regressor vector then is defined as 
\be\label{wzu}
\zeta = F(s) I \omega
\ee
where $I$ is the $2n$ by $2n$ identity matrix, the high frequency gain of $F(s)$ is unity, and $F(s)$ and $\ell$ chosen so that
\be
 W^\prime_f(s)  \triangleq W_e^\prime(s)F(s)^{-1}
\ee
is SPR and $W_f(s) = k_m W_f^\prime(s)$.
In addition to the adaptive parameters in the control law however another adaptive parameter $k_\chi(t)$ is included whose parameter error is defined as
\be
\psi \triangleq k_\chi(t) - k_p
\ee
with an update law shortly to be defined. The error equations for this system then are constructed as
\begin{align}
e_a& \triangleq e_y + W^\prime_f(s) \left(k_\chi e_\chi  - e_a\zeta^T \zeta\right) \label{eau}\\
e_\chi& \triangleq \theta^T \zeta - F(s)\theta^T \omega\label{euu}.
\end{align}
The update law for the adaptive parameters is then chosen as
\begin{align}
\dot \theta(t) =& -\gamma \text{sign} (k_p)  e_a \zeta \label{tbone1} \\
\dot k_\chi(t) = &-\gamma e_a e_\chi \label{tbone2}.
\end{align}
\begin{thm}
Following Assumptions 1, 2 and 3, with $\ell$ chosen such that $W^\prime_f(s)$ is SPR, the plant in \eqref{iop} with the reference model in \eqref{crm1:ref}, controller in \eqref{u1} and update law in \eqref{tbone1}--\eqref{tbone2} are globally stable with the model following error $e_y$ asymptotically converging to zero.
\end{thm}
\begin{proof}
The entire proof would come in 4 parts just as in the proof of Theorem 5. We however only present a detailed proof of step 1 and then briefly present the other 3 steps. 
\subsubsection*{Step 1} The boundedness of $e_a$, $\phi$ and $\psi$ are now addressed. First consider the representation of \eqref{eau} 
\ben\begin{split}
e_a  =&W^\prime_e(s) k_p \phi^T \omega +W^\prime_f(s) \left(  k_\chi e_\chi-  e_a \zeta^T \zeta\right) \\ &+W^\prime_f(s) (k_p e_\chi - k_p e_\chi)
\end{split}\een
where $k_pe_\chi$ has been added and subtracted from.  Expanding $k_p e_\chi$, $W_f^\prime(s)$ and $\phi$ we have
\ben\begin{split}
e_a  = & W^\prime_e(s) k_p (\theta-\theta^*)^T \omega +W^\prime_f(s) \left( \psi e_\chi  -  e_a \zeta^T \zeta\right) \\ &+W^\prime_e(s) k_p F(s)^{-1}\left(  \theta^{T} \zeta - F(s)\theta^{T}\omega\right).
\end{split}\een
Canceling like terms in $\theta^T\omega$, and adding and subtracting the term $W_f^\prime(s)\theta^{*T} \zeta$ the expression reduces to
\be\label{lastea}
e_a  =  W^\prime_f(s) \left( k_p \phi^T\zeta   + \psi e_\chi  -  e_a \zeta^T \zeta\right) +\delta(t)
\ee
where $\delta$ is an exponentially decaying term defined as
\ben
\delta(t)=W_f^\prime(s) k_p \left(\widebar\theta^{*T} \left( F(s)  - F(s)\right) I \widebar\omega\right).
\een
Therefore, if the filter $F(s)$ is chosen to have the same initial conditions when constructing $\zeta$ and $e_\chi$, then $\delta=0$ for all time.  For this reason we ignore the affect of choosing different filter initial conditions. The interested reader can see how one can prove stability in augmented error approaches where $\delta(0)\neq 0$ \cite[pg. 213]{annbook}, with the addition of an extra term in the  Lyapunov function. 
Given that $\theta^*$ is constant and the following holds.
Now consider a non--minimal representation of $e_a$ from \eqref{lastea} as
\be\label{eanmu}\begin{split}
\dot e_{an} &= A_{e} e_{an} + b_{an}\left( k_p \phi^T  \zeta + \psi e_\chi- e_a \zeta^T \zeta\right) \\ e_a & = c^T_{an} e_{an}
\end{split}\ee
where
\be
c_{an}^T(sI-A_e)^{-1} b_{an} \triangleq W^\prime_f(s).
\ee
Given that $W^\prime_f(s)$ is SPR, there exists a ${P_a=P_a^T>0}$ such that
\be
A_e^T P_a+P_a A_{e}= - Q_a \text{ and } P_a b_{an} = c_{an}.
\ee
where ${Q_a=Q_a^T>0}$.

Consider the Lyapunov candidate
\be\label{lyap}
V=e_{an}^T P_{a}e_{an} +  \frac{\phi^T \phi}{\gamma \abs{k_p}} + \frac{\psi^2}{\gamma}
\ee
Differentiating along the system dynamics in \eqref{eanm} and substitution of the tuning law from \eqref{tdbar} results in
\be
\dot V \leq - e_{an}^TQ_a e_{an} - 2  e_a^2 \zeta^T \zeta.
\ee
Therefore, $e_{an},\theta,k_\chi \in \mathcal L_\infty$ and $e_{an},\dot{\theta} \in\mathcal L_2$.
\subsubsection*{Step 2} Given that $\phi$ is bounded, then \eqref{eq:x} can grow at most exponentially. 
\subsubsection*{Step 3} The only difference between the $k_p$ known and unknown case is the addition of $k(t)$ in the feedforward loop and $k_\chi(t)$ in the augmented error. Then, if we assume that signals in the system grow in an unbounded fashion and using the results from \eqref{sim} it immediately follows that
\be\label{sim2}
\begin{split}
\sup_{\tau\leq t} \abs{y(\tau)}\sim \sup_{\tau\leq t}\norm{\omega_1(\tau)} \sim \sup_{\tau\leq t}\norm{\omega_2(\tau)}  \ldots \\ \sim \sup_{\tau\leq t}\norm{\widebar\omega}\sim \sup_{\tau\leq t}\norm{\widebar\zeta}\sim \sup_{\tau\leq t}\norm{\omega} \ldots \\ \sim \sup_{\tau\leq t}\norm{\zeta} \sim \sup_{\tau\leq t} \abs{u(\tau)} 
\end{split}
\ee
where $\widebar\zeta$ and $\widebar\omega$ are defined in \eqref{wz} and \eqref{wbaro} respectively.
\subsubsection*{Step 4} Given that $\dot{\widebar\theta}\in\mathcal L_2$ and $F(s)$ is stable the following holds
\be\label{euou}
e_\chi(t) = o\left[\sup_{\tau\leq t} \norm{\omega(\tau)}\right].
\ee
Then, following the same steps as in Step 4 from the proof of Theorem 5 we can conclude that
\be
y(t) = o\left[\sup_{\tau\leq t} \norm{\omega(\tau)}\right].
\ee
This contradicts \eqref{sim2} and therefore all signals are bounded. Furthermore, from \eqref{eanmu} it now follows that $\dot e_{an}$ is bounded and given that $e_{an}\in \mathcal L_2$, from Step 1, it follows that $e_{an}$ asymptotically converges to zero and therefore $\lim_{t\to\infty}e_a(t)=0$. From \eqref{euou} it follows that $e_\chi$ asymptotically converges to zero. Therefore, $\lim_{t\to\infty} e_y(t) = 0$.
\end{proof}

\section{Conclusion}
This work shows that with the introduction of CRMs the adaptive system can have improved transient performance in terms of reduction of the $\mathcal L_2$ norm of the model following error. Similar to previous work in\cite{gib13tranA}, bounds on derivatives of key signals in the system, and trade--off between transients and learning remain to be addressed and is the subject of on--going investigation.

\section*{Acknowledgment}
This work was supported by the Boeing Strategic University Initiative.

\bibliographystyle{IEEEtran}
\bibliography{ref}

\appendices
\section{Norm of $e_\chi(t)$} \label{app1}
In this Appendix we compute the $\mathcal L_2$ norm of $e_\chi(t)$. The expression in \eqref{chibound} is equivalent to studying the equation
\be\label{inter1}
e_\chi(t) = \left[\widebar\theta^T(t) - F(s) \widebar\theta^T(t) F(s)^{-1}\right] F(s) I \widebar\omega(t) 
\ee
Given the definition of $F(s)$ in \eqref{defs} we have that
\be
F(s) \widebar\theta^T(t) F(s)^{-1}= \widebar\theta^T(t) - \frac{1}{s+f_1} \dot{\bar \theta}^T(t). 
\ee
This allows \eqref{inter1} to be rewritten as
\be\label{eq74}
e_\chi(t)=\frac{1}{s+f_1} \dot{\bar \theta}^T(t) \frac{1}{s+f_1} I \widebar\omega(t).
\ee
This is analyzed in 3 parts
\be
\abs{e_\chi(t)} \leq \chi_1(t) + \chi_2(t) + \chi_3(t)
\ee
where 
\begin{align}
 \chi_1(t) =&e_\chi(0)\Phi_f(t,0) \label{c1} \\
 \chi_2(t)=&\int_0^t \norm{\dot{\bar\theta}(\tau)} \Phi_f(t,\tau)  e_\chi(0) \Phi_f(\tau,0) d\tau \label{c2}\\ 
 \chi_3(t) =&\int_0^t \norm{\dot{\bar\theta}(\tau)}  \Phi_f(t,\tau)   \int_0^\tau\Phi_f(\tau,z) \norm{\widebar\omega(z)} dz d\tau\label{c3}
 \end{align} 
and
\be
\Phi_f(t,\tau) =\exp{(-f_1(t-\tau))}.
\ee
Then the $\mathcal L_2$ norm of $e_\chi(t)$ is obtained as
\be
\norm{e_\chi(t)}_{\mathcal L_2}^2 \leq 3 \sum_{i=1}^3\int_0^\infty \chi_i^2(\tau) d\tau. \ee
Squaring and integrating \eqref{c1} we have that
\be
\int_0^\infty\chi_1^2(\tau) d\tau  \leq \frac{e_\chi^2(0)}{2 f_1}. 
\ee
%
%
%
%
Notice that $\Phi_f(t,0)=\Phi_f(t,\tau) \Phi_f(\tau,0)$ is not a function of $\tau$ and therefore can be pulled out of the integral in \eqref{c2} resulting in
\be
\chi_2(t) \leq e_\chi(0) \Phi_f(t,0)  \int_0^t \norm{\dot{\bar\theta}(\tau)} d\tau.
\ee
Using Youngs inequality \ben\int_0^t \norm{\dot{\bar\theta}(\tau)} d\tau \leq \left( \int_0^t 1^2 d\tau \right)^{1/2} \left(\int_0^t \norm{\dot{\bar\theta}(\tau)}^2 d\tau \right)^{1/2}\een and therefore
\be
\chi_2(t) \leq e_\chi(0) \sqrt{t} \Phi_f(t,0) \norm{\dot{\bar\theta}(\tau)}_{\mathcal L_2}.
\ee
Squaring the result above and integrating we have that
\be
\int_0^\infty\chi_2^2(\tau) d\tau \leq \frac{e_\chi(0)^2}{4 f_1^2} \norm{\dot{\bar\theta}(\tau)}_{\mathcal L_2}^2
\ee

Integrating the inner integral in \eqref{c3} we have that
\be
\chi_3(t) \leq \frac{ \norm{\widebar\omega(t)}_\infty }{f_1} \int_0^t \norm{\dot{\bar\theta}(\tau)}  \Phi_f(t,\tau) (1- \Phi_f(\tau,0) ) d\tau.
\ee
Noting that ${[1-\Phi_f(t,0)] \leq 1}$ for all $t$ the above simplifies to
\be\label{abover2}
\chi_3(t) \leq \frac{ \norm{\widebar\omega(t)}_\infty }{f_1} \int_0^t \norm{\dot{\bar\theta}(\tau)}  \Phi_f(t,\tau) d\tau.
\ee
Using Young's Inequality we have that
\be\begin{split}
\int_0^t \norm{\dot{\bar\theta}(\tau)}  \Phi_f(t,\tau) d\tau \leq &\left ( \int_0^t {\Phi_f(t,\tau)} d\tau\right)^{1/2}  \\& \cdot\left ( \int_0^t {\Phi_f(t,\tau)} \norm{\dot{\bar\theta}(\tau)}^2  d\tau\right)^{1/2}
\end{split}\ee
and bounding the first integral term we have that
\be\begin{split}\label{abover}
\int_0^t \norm{\dot{\bar\theta}(\tau)}  \Phi_f(t,\tau) d\tau \leq &\frac{1}{\sqrt{f_1}} \left ( \int_0^t {\Phi_f(t,\tau)} \norm{\dot{\bar\theta}(\tau)}^2  d\tau\right)^{1/2}.
\end{split}\ee
Substitution of \eqref{abover} into \eqref{abover2}, squaring  and integrating we have that
\be
\int_0^\infty\chi_3^2(\tau)d\tau \leq \frac{ \norm{\widebar\omega(t)}_\infty^2 }{f_1^3} \norm{\dot{\bar\theta}(t)}_{\mathcal L_2}^2.
\ee

\section{Norm of $e_a(t)$}\label{app2}
Noting that $\frac{a}{1 + b} \leq a$ for all $a,b\geq0$, $e_y$ in \eqref{ea} can be bounded as
\be
\abs{e_y(t)} \leq \abs{e_a(t)} + \abs{W_f(s)e_\chi(t)}.
\ee
From \eqref{eq74} and the definition of $W_f(s)$ in \eqref{above11} the filtered error state $e_\zeta$ from \eqref{echi} satisfies the following equality
\be
e_\zeta(t) = W_e(s) \dot {\bar \theta}^T(t) \frac{1}{s+f_1} I \widebar\omega(t).
\ee
We will also make use of the fact that there exist an $m\geq1$ such that
\be
\exp{(A_\ell t)} \leq m \exp{(-\mu t)}.
\ee
$e_\zeta$ is analyzed in 3 parts just as we did with $e_\chi$
\be
\abs{e_\zeta(t)} \leq \zeta_1(t) + \zeta_2(t) + \zeta_3(t)
\ee
where 
\begin{align}
 \zeta_1(t) =&e_\zeta(0) m \Phi_\mu(t,0) \label{d1} \\
 \zeta_2(t)=& e_\chi (0) m\int_0^t \norm{\dot{\bar\theta}(\tau)} \Phi_\mu (t,\tau)   \Phi_f(\tau,0) d\tau \label{d2}\\ 
 \zeta_3(t) =& m \int_0^t \norm{\dot{\bar\theta}(\tau)}  \Phi_\mu (t,\tau)   \int_0^\tau\Phi_f(\tau,z) \norm{\widebar\omega(z)} dz d\tau\label{d3}
 \end{align} 
and then the $\mathcal L_2$ norm of $e_\zeta(t)$ is obtained as
\be
\norm{e_\zeta(t)}_{\mathcal L_2}^2 \leq 3 \sum_{i=1}^3\int_0^\infty \zeta_i^2(\tau) d\tau. \ee
Squaring and integrating \eqref{d1} we have that
\be
\int_0^\infty\zeta_1^2(\tau) d\tau  \leq \frac{m^2 e_\zeta^2(0)}{2 \mu}. 
\ee
%
%
%
%
%
Using Young's inequality the integral in \eqref{d2} can be upper bounded by 
$ \left( \int_0^t \Phi^2_\mu(t,\tau)\Phi^2_f(\tau,0) d\tau \right)^{1/2}  \norm{\dot{\bar\theta}(t)}_{\mathcal L_2}$ and after computing the integral in the first term reduces to
$\left(\frac{ \Phi_f(2t,0)-\Phi_\mu(2t,0)}{2(\mu-f_1)}\right)^{1/2}  \norm{\dot{\bar\theta}(t)}_{\mathcal L_2}$. Using this, squaring and integrating \eqref{d2} we have that
\be
\int_0^\infty\zeta_2^2(\tau) d\tau \leq \frac{m^2e_\chi(0)^2}{4 \mu f_1} \norm{\dot{\bar\theta}(\tau)}_{\mathcal L_2}^2
\ee

Integrating the inner integral in \eqref{d3} we have that
\be
\zeta_3(t) \leq \frac{ m \norm{\widebar\omega(t)}_\infty }{f_1} \int_0^t \norm{\dot{\bar\theta}(\tau)}  \Phi_\mu(t,\tau) (1- \Phi_f(\tau,0) ) d\tau.
\ee
Noting that ${[1-\Phi_f(t,0)] \leq 1}$ for all $t$ the above simplifies to
\be\label{abover4}
\zeta_3(t) \leq \frac{m \norm{\widebar\omega(t)}_\infty }{f_1} \int_0^t \norm{\dot{\bar\theta}(\tau)}  \Phi_\mu(t,\tau) d\tau.
\ee
Using Young's Inequality we have that
\be\begin{split}
\int_0^t \norm{\dot{\bar\theta}(\tau)}  \Phi_\mu(t,\tau) d\tau \leq &\left ( \int_0^t {\Phi_\mu(t,\tau)} d\tau\right)^{1/2}  \\& \cdot\left ( \int_0^t {\Phi_\mu(t,\tau)} \norm{\dot{\bar\theta}(\tau)}^2  d\tau\right)^{1/2}
\end{split}\ee
and bounding the first integral term we have that
\be\begin{split}\label{abover3}
\int_0^t \norm{\dot{\bar\theta}(\tau)}  \Phi_f(t,\tau) d\tau \leq &\frac{1}{\sqrt{\mu}} \left ( \int_0^t {\Phi_\mu(t,\tau)} \norm{\dot{\bar\theta}(\tau)}^2  d\tau\right)^{1/2}.
\end{split}\ee
Substitution of \eqref{abover3} into \eqref{abover4}, squaring  and integrating we have that
\be
\int_0^\infty\zeta_3^2(\tau)d\tau \leq \frac{m^2 \norm{\widebar\omega(t)}_\infty^2 }{\mu f_1^2} \norm{\dot{\bar\theta}(t)}_{\mathcal L_2}^2.
\ee

%
\end{document}